\documentclass{article}
\usepackage{graphicx}
\usepackage{amsmath}
\usepackage{amsfonts}
\usepackage{amssymb}
\usepackage{color}

\newtheorem{theorem}{Theorem}

\newtheorem{corollary}[theorem]{Corollary}

\newtheorem{definition}[theorem]{Definition}
\newtheorem{example}[theorem]{Example}

\newtheorem{lemma}[theorem]{Lemma}

\newtheorem{proposition}[theorem]{Proposition}
\newtheorem{remark}[theorem]{Remark}

\newenvironment{proof}[1][Proof]{\textbf{#1.} }{\ \rule{0.5em}{0.5em}}

\newcommand{\norm}[1]{\|#1\|}
\newcommand{\n}{\mathbb{N}}

\newcommand{\del}[3]{\delta_{{#1},{#2},{#3}}}
\newcommand{\f}[1]{#1\circ #1 \circ \ldots \circ #1}
\newcommand{\stack}[2]{\underset{#1
\,\,\text{times}}{\underbrace{#2}}}

\newcommand{\cart}{i_1,i_2,\ldots,i_{k-1}}

\begin{document}

\title{ 2-Homogeneous Polynomials and Maximal Subspaces}
\author{Carlos A. S. Soares - carlos.soares@ufjf.edu.br\\Universidade Federal de Juiz de Fora - Departamento de Matematica/ICE\\Cidade Universitaria - Juiz de Fora - MG - Brasil \\ 30036-330} \maketitle
\begin{abstract}
In this paper we study the maximal subspaces of continuous n-homoge-
neous polynomials on complex and real non separable Banach spaces. In the real case we will prove that if $P:X\rightarrow
\mathbb{R}$ is a 2-homogeneous polynomial and  if
there exist a k-dimensional $P$-maximal subspace then every
$P$-maximal subspace is k-dimensional.
\end{abstract}

\bigskip



\bigskip

The zero-set of a continuous homogeneous polynomial has been
studied by many authors as we can see in \cite{a1,p1,p2}. In
particular in \cite{p2} the authors show that for the complex
case, every continuous homogeneous polynomial of any degree is
identically zero on an infinite dimensional subspace.
In \cite{p1} and \cite{ca} conditions are given on a complex
Banach space $X$ such that every 2-homogeneous polynomial
$P:X\rightarrow\mathbb{C}$ is identically zero on a non separable
subspace. 

We recall that an n-homogeneous polynomial on a complex Banach
space $X$ is a mapping $P:X\rightarrow\mathbb{C}$ such that there
exists a symmetric n-linear mapping $\overset{\vee}{P}:X\times
X\times \ldots \times X\rightarrow \mathbb{C}$ with
$P(x)=\overset{\vee}{P}(x,x,\ldots,x)$ and $P$ will be said to be
continuous if there is a constant $k$ such that $P(x)\le k||x||^n$
for all $x$ in $X$. In this paper all n-homogeneous polynomial
considered will be continuous. The vector space of all continuous
n-homogeneous polynomials $P:X\rightarrow\mathbb{C}$ will be
denoted by $\mathcal{P}(^nX)$. We also observe that a polynomial
$P\in \mathcal{P}(^nX)$ will be said to be a nuclear polynomial if
there is a bounded sequence $(\varphi_i)\in X^{\ast}$ and
$(\lambda_i) \in l_1$ such that
$$P(x)=\sum_{i=1}^\infty \lambda_i \varphi_i(x)^n \;\; \mbox{for
every}\;\; x\in X.$$

We start with a simple observation, that if
$P:X\rightarrow\mathbb{C}$ is an n-homogeneous polynomial and
$H\subset X$ is a 2 dimensional subspace, then there is $z\in H$,
$z\ne 0$ such that $P(z)=0$. Indeed if $P(y)$ and $P(x)$ are not
zero with $x$ and $y$ in $H$ independent, then
$$p(\alpha)=P(x+\alpha
y)=\sum_{j=0}^{n}\binom{n}{j}\alpha^{n-j}\overset{\vee}{P}(x^{j},y^{n-j})$$
is a complex polynomial and taking $\alpha_1$ one of its roots we
have $P(x+\alpha_1 y)=0$.

\begin{definition} Let X be a Banach space and
$P:X\rightarrow\mathbb{K}$ a non zero n-homogeneous polynomial . A
subspace $M\subsetneqq X$ will be called P-maximal if $P|_{M}=0$
and if whenever $M_{1}$ is a subspace such that $P|_{M_1}=0$ with
 $M\subset M_{1}\subset X,$ then $M_{1}=M.$

\end{definition}

\bigskip
If $P:X\rightarrow \mathbb{K}$ is a non zero n-homogeneous
polynomial and $M$ is a subspace such that $P|_M=0$ it is simple
 to get a P-maximal subspace containing $M$. For an infinite
dimensional complex space we can prove a bit more.

\begin{theorem}\label{carlos}
Let X be an infinite dimensional complex Banach space and
\\ $P:X\rightarrow \mathbb{C}$ be a non zero 2-homogeneous polynomial. If $M \subset X$ is a subspace with $P|_M=0$, then there is an
infinite dimensional $P$-maximal subspace $H$ such that $M \subset
H.$

\end{theorem}

\begin{proof}Let $\mathcal{S}=\{N: N$ is a subspace of $X$ and $M \subset N$ and
$P|_N=0\}.$ Partially ordering $S$ by inclusion and using Zorn`s
Lemma we obtain a maximal element $H.$ It is clear that $H\ne
\{0\}$ and that $H$ is a $P$-maximal subspace. Suppose that $H$ is
finite dimensional and let $H=[e_1,\ldots ,e_n].$ For each
$i=\{1,2,\ldots ,n\}$, take $\varphi_{i}\in X^{\ast}$ given by
$\varphi_{i}(x)=\overset{\vee}{P}(e_{i},x).$ So
$\mathcal{K}=\cap_{i=1}^{n}\varphi_{i}^{-1}(0)$ is an infinite
dimensional subspace and $M\subset \mathcal{K}$. Now we choose
linearly independent vectors $y$ and $z$ in $\mathcal{K}$ such
that $\mathcal{M}=[e_1,e_2, \ldots ,e_n,y,z]$ is an
(n+2)-dimensional subspace. If $P(y)=0$ or $P(z)=0$ we have
$P|_{[e_1,e_2,\ldots ,e_n,y]}=0$ or $P|_{[e_1,e_2,\ldots
,e_n,z]}=0$ respectively, because, for example, $P(\sum_{i=1}^n
t_ie_i+t_{n+1}y)=P(\sum_{i=1}^n
t_ie_i)+2\overset{\vee}{P}(\sum_{i=1}^n
t_ie_i,t_{n+1}y)+P(t_{n+1}y)=2\sum_{i=1}^n
t_it_{n+1}\overset{\vee}{P}(e_i,y)=0$ for every $t_1,t_2,\ldots
,t_n,t_{n+1} \in \mathbb{C}.$ That is a contradiction with the
maximality of $H$. If $P(y)$ and $P(z)$ are non zero we can obtain
$\alpha \in \mathbb{C}$ such that if $w=y+\alpha z$ we have
$P(w)=0$ and consequently in that case $P|_{[e_1,e_2,\ldots
,e_n,w]}=0$, and it is again a contradiction with the maximality
of $H$.
\end{proof}

\bigskip

Note that the proof above shows that every P-maximal subspace is
an infinite dimensional subspace. In the real case we are going to
prove that if a 2-homogeneous polynomial has a finite dimensional
P-maximal subspace, then every P-maximal subspace is finite
dimensional.

\begin{corollary}
Let $X$ be either complex $c_0$ or $l_p$. If $P:X
\rightarrow\mathbb{C}$ is a 2-homogeneous polynomial, then every
P-maximal subspace contains an isomorphic subspace that is
complemented in $X$.
\end{corollary}
\begin{proof} It is well known that every infinite dimensional
subspace of $X$ contains an isomorphic subspace that is
complemented in $X$.
\end{proof}

\begin{theorem} Let $X$ be a real Banach space and $P:X \rightarrow\mathbb{R}$
be a 2-homogeneous polynomial, with $k$-dimensional P-maximal
subspace $M$. If $H\subset X$ is an $m$-dimensional subspace such
that $P|_H=0$ then $m\le k$.
\end{theorem}
\begin{proof} Suppose that there is a $(k+1)$-dimensional subspace $N\subset H$. Let $\{w_1,w_2,\ldots ,w_j\}$ be a basis of $J=M\cap N$ and $\{w_1,w_2,\ldots ,w_j,w_{j+1},\ldots ,w_k\}$ be a basis
 of $M$. Since $J\subset N$, there is a $(k-j+1)$-dimensional
 subspace $I\subset N$ such that
 \begin{equation}\label{eq10}
        N=J\oplus I.
\end{equation}
Now, for each $i=\{1,2,\ldots ,k-j\}$, take $\varphi_{i}:I
\rightarrow\mathbb{R}$ given by
\begin{equation}\label{eq20}
       \varphi_{i}(x)=\overset{\vee}{P}(w_{j+i},x)
\end{equation}
 Since $dim(I)=k-j+1$ we have
$dim(\cap_{i=1}^{k-j}\varphi_{i}^{-1}(0))\geq 1$. Thus we can take
$w\ne 0\in \cap_{i=1}^{k-j}\varphi_{i}^{-1}(0)$ and we observe
that $w\notin M.$ Indeed, if $w\in M$ we have $w\in J$ and
consequently $w\in J\cap I$, a contradiction by [\ref{eq10}]. Now
note that by [\ref{eq20}] $\overset{\vee}{P}(w_{j+i},w)=0$ for
every $i=1,2,\ldots ,k-j$. Since $w\in N$ and $P|_N=0$ we also
have $\overset{\vee}{P}(w_{i},w)=0$ for every $i=1,2,\ldots ,j$.
Therefore $P|_{M\oplus [w]}=0$ but that is a contradiction to the
maximality of $M$.
\end{proof}

\begin{corollary}Let $X$ be a real Banach space and $P:X \rightarrow\mathbb{R}$
be a 2-homogeneous polynomial. If $P$ is identically zero on an
infinite dimensional subspace, then every P-maximal subspace is
infinite dimensional.
\end{corollary}
\begin{example}
For each natural number k let $P_k:l_2 \rightarrow\mathbb{R}$ be
given by
\begin{equation}\label{eq50}
       P_k(x(i))=-\sum_{i=1}^k x(i)^2+\sum_{i=k+1}^\infty x(i)^2
\end{equation}
For every $j=1,2,\ldots,k$ let $x_j\in l_2$ be given by

      $$ x_j(i)=\left\{
       \begin{array}{ll}
         1 & if \,\,\, i=j \,\,\, or\,\,\, i=j+k, \\
         0 & otherwise. \\
       \end{array}
       \right.$$

It is clear that $M_k=[x_1,x_2,\ldots,x_k]$ is a k-dimensional
subspace and $P|_{M_k}=0$. We claim that $M_k$ is a $P_k$-maximal
subspace. Suppose $y=y(i)\in l_2$ such that $P_k|_{M_k+[ y]}=0$.
Note that
\begin{equation}\label{eq60}
       \overset{\vee}{P_k}(x_j,y)=-\sum_{i=1}^k x_j(i)y(i)+\sum_{i=k+1}^\infty x_j(i)y(i)=-y(j)+y(j+k)
\end{equation}
Since $\overset{\vee}{P_k}(x_j,y)=0$ we obtain $y(i)=y(i+k)$ for
every $i=1,2,\ldots,k.$ We also have
\begin{equation}\label{eq70}
        P_k(y(i))=-\sum_{i=1}^k y(i)^2+\sum_{i=k+1}^{2k} y(i)^2+\sum_{i=2k+1}^\infty
        y(i)^2=\sum_{i=2k+1}^\infty
        y(i)^2
\end{equation}
Since $P_k(y)=0$ we obtain $y(i)=0$ for every $i>2k$ and so
$y=\sum_{j=1}^k y(j)x_j \in M_k$.

As a consequence of the last theorem we get that every
$P_k$-maximal subspace is a $k$-dimensional subspace.
\end{example}

\begin{remark}
It is easy to see that the Theorem 2 does not hold in the real
case when $n$ is even. In the general case, when $n$ is odd, the
best result is shown in \cite{a2}, that is, if $X$ is an real
infinite dimensional Banach space and $P:X \rightarrow\mathbb{R}$
is an n-homogeneous polynomial, then for each natural number $k$
there exist a $k$-dimensional subspace $M_k$ such that
$P|_{M_k}=0.$ Nevertheless it is shown in \cite{a3} that in every
real separable Banach space X there exists an n-homogeneous
polynomial $P$ ($n>1$ arbitrary odd number) such that $P$ is not
identically zero on $M$ if $M\subset X$ is an infinite dimensional
subspace.

\end{remark}
\bigskip

We now provide an example to show that for $n$-homogeneous
polynomials $(n\geq 3)$ the previous theorem does not hold.

\begin{example}
Take $X=l_2$ the standard real space. For each natural number $k$
we define the n-homogeneous polynomial $P_k:l_2
\rightarrow\mathbb{R}$ given by
\begin{equation}\label{eq100}
       P_k(x)=x(1)^{n-2}\sum_{i=k+1}^\infty x(i)^2.
\end{equation}

It is clear that $M=\{ x\in l_2;\,x(1)=0\}$ is a 1-codimensional
$P_k$-maximal subspace. We will prove that
\begin{equation}\label{eq200}
       M_k=\{x\in l_2; \,x(i)=0\,\,\mbox{if}\,\,i>k\}
\end{equation}
is a $P_k$-maximal subspace and consequently for each $P_k$ there
are only two $P_k$-maximal subspaces, one 1-codimensional and the
other k-dimensional. Suppose $y\in l_2$ such that $P_k|_{M_k\oplus
[y]}=0.$ Since $P_k(y)=0$ and $y\notin M_k$ we have $y(1)=0.$ Now,
let $x\in l_2$ given by $x(1)=1$ and $x(i)=0$ if $i>1.$ Thus,
since $x\in M_k$ we have $P_k(x+y)=0$, that is, $(x+y)(1)=0$ or
$\sum_{i=k+1}^\infty (x+y)(i)^2=0.$ As $(x+y)(1)=1$ we must have
$\sum_{i=k+1}^\infty (x+y)(i)^2$ and so $y(i)=0$ for every $i>k$
and that implies $y\in M_k$, a contradiction.
\end{example}

It is important to observe that in the proof of theorem, we know
the size of $\cap _{i=1}^{k}\varphi_{i}^{-1}(0)$. This motivates
the following definition ( cf \cite{p1} ).

\begin{definition} A Banach space $X$ is said to have the nontrivial
intersection property ( $X$ has the NIP ) if for every countable
family $\{\varphi_i\}\in X^{\ast}$ we have $\cap
_{i=1}^{\infty}\varphi_{i}^{-1}(0)\ne \{0\}.$
\end{definition}
\begin{remark} It is simple, by the Hahn-Banach theorem to show that
if $X$ is a separable space then there exists a countable family
$\{\varphi_i\}\in X^{\ast}$ such that $\cap
_{i=1}^{\infty}\varphi_{i}^{-1}(0)=\{0\}$ and consequently if $X$
has the NIP we have $X$ is a non separable space.
\end{remark}
\begin{proposition} Let $X$ be a non separable Banach space. Then
$X$ has the NIP if and only if for every countable family
$\{\varphi_i\} \textcolor{red}{\subset} X^{\ast}$, we have $\cap
_{i=1}^{\infty}\varphi_{i}^{-1}(0)$ is a non separable subspace.
\end{proposition}

\begin{proof} It is a simple consequence of the previous remark.
\end{proof}

\begin{proposition} Let $X$ be a non separable real Banach space. The
following assertions are equivalent:

(i) $X$ has the NIP

(ii) For every nuclear polynomial $P\in \mathcal{P}(^2X)$ there is
an $x_P \ne 0$ in $X$ such that $P(x_P)= 0$

(iii) For every nuclear polynomial $P\in \mathcal{P}(^2X)$ there
is a non separable subspace $M\subset X$ such that $P|_M=0$.

(iv) If $P\in \mathcal{P}(^2X)$ is a nuclear polynomial then every
P-maximal subspace is an infinite dimensional subspace.
\end{proposition}

\begin{proof}

(i)$\Rightarrow$(ii) Let $P(x)=\sum_{i=1}^\infty \lambda_i
\varphi_i(x)^2$ with $(\lambda_i) \in l_1$ and $(\varphi_i)\in
X^\ast$ a bounded sequence. Since $X$ has the NIP, we have
$\cap_{i=1}^\infty \varphi_i^{-1}(0) \ne \{0\}$; that is, there is
$x_P\ne 0$ such that $\varphi_i(x_P)=0$ for every $i=1,2,\ldots$
and so $P(x_P)=0$.

(ii)$\Rightarrow$(iii) If $P(x)=\sum_{i=1}^\infty \lambda_i
\varphi_i(x)^2$ is a nuclear polynomial, let $Q(x)\in
\mathcal{P}(^2X)$ be given by $Q(x)=\sum_{i=1}^\infty
\frac{\varphi_i(x)^2}{i^2\|\varphi_i \|}$.  Since $Q$ is a nuclear
polynomial there is $x_Q \in X\backslash \{0\}$ such that
$Q(x_Q)=0$; that is, $x_Q\in \cap_{i=1}^\infty \varphi_i^{-1}(0)$
and by the previous proposition we have $\cap_{i=1}^\infty
\varphi_i^{-1}(0)$ is a non separable subspace and the result
follows.

 (iii)$\Rightarrow$(iv) It is enough to note that if  $P\in
 \mathcal{P}(^2X)$is a nuclear polynomial and $M$ is a  finite
 dimensional subspace, by the Theorem [4] every $P$-maximal
 subspace is finite dimensional.

 (iv)$\Rightarrow$(i)
 Let $\{\varphi_i\}$ a countable
family in $X^\ast$ and define $P:X\rightarrow \mathbb{R}$ by
$$P(x)=\sum_{i=1}^\infty \frac{\varphi_i(x)^2}{i^2\|\varphi_i
\|}.$$ Thus $P$ is a nuclear polynomial and there exists an
infinite dimensional subspace $M$ such that $P|_M=0$ and since
$M\subset \cap_{i=1}^\infty \varphi_i^{-1}(0)$ we have  $ \cap
_{i=1}^{\infty}\varphi_{i}^{-1}(0)\ne 0$.
\end{proof}

\bigskip
In \cite{a1} the following theorem is proved.
\begin{theorem} Let $X$ be a real non separable space that has the NIP and $P\in \mathcal{P}(^2X)$. Then there exists an infinite
dimensional subspace $M$ such that $P|_M=0$.
\end{theorem}

\begin{corollary}
Let $X$ be a real non separable space that has the NIP and $P\in
\mathcal{P}(^2X)$. Then every P-maximal subspace is infinite
dimensional.
\end{corollary}
\begin{proof}
Easy consequence of Theorems 4 and 13.
\end{proof}

\bigskip
For the complex case, using the same argument as Theorem 2, we can
obtain non separable maximal subspaces.

\bigskip

\begin{proposition} Let X be a complex non separable space and let
$P:X\rightarrow\mathbb{C}$ be a 2 homogeneous polynomial. Suppose
that $X$ has the NIP. Then every P-maximal subspace is non
separable.
\end{proposition}

\bigskip

\begin{remark} Now we observe that if $X$ is a Banach space and $M$ is a
closed subspace, since $(X/M)^\ast$ is isometric to $M^\perp$ we
have $X/M$ has the NIP if and only if for every countable family
$\{\varphi_i \}$ in $X^\ast$ with $\varphi_i |_M=0$ there exists
$x\notin M$ such that $x\in \cap
_{i=1}^{\infty}\varphi_{i}^{-1}(0)$; that is, we have $\cap
_{i=1}^{\infty}\varphi_{i}^{-1}(0)$ is a non separable subspace.
As we know this is the case for $l_\infty$ and $c_0$ ( cf
\cite[page 61]{caro} ) we have $l_\infty /c_0$ has the NIP and
consequently if $P:l_\infty/c_0 \rightarrow \mathbb{K}$ is a
2-homogeneous polynomial, in the real case we have every
$P$-maximal subspace is an infinite dimensional subspace and in
the complex case there exists a non separable $P$-maximal
subspace.
\end{remark}

We now recall that in \cite[page 58]{li} the following result is
shown.
\begin{lemma}
Let $(y_n)$ be a sequence in $l_\infty$ and constants $C_1$ and
$C_2$ such that $$C_1\sup_i|\lambda_{i}|\le
\norm{\sum_{i=1}^\infty \lambda_{i} y_i}\le C_2
\sup_i|\lambda_{i}|$$ for every scalar sequence $(\lambda_{i})$
with $\lim \lambda_{i}=0$. Then there is a subsequence $(y_{n_k})$
of $(y_n)$ such that for every $(\lambda_k) \in l_\infty$  we have
$$(C_1/2)\sup_k |\lambda_{k}| \le \|\sum_{k=1}^\infty \lambda_{k} y_{n_k}\|\le
C_2 \sup_k|\lambda_{k}|$$
\end{lemma}

\begin{theorem} Let $M\subset l_{\infty}$ be a subspace isomorphic
to $c_0$. Then there exists a subspace $N\subset M$, and a
subspace $H\subset l_\infty$ containing $N$, and an isomorphism
$T:l_\infty \rightarrow H$ such that $T|_{c_0}:c_0\rightarrow N$
is an isomorphism. Moreover $H/N$ has the NIP.
\end{theorem}
\begin{proof} Using the previous lemma we can get a sequence $(y_k)$ in $M$
and constants $C_1$ and $C_2$ such that
$$C_1\sup |\lambda_k|\le \sup_i |\sum_{k=1}^\infty \lambda_k
y_k(i)|\le C_2\sup |\lambda_k|\;\; \mbox{if}\;\; (\lambda_k)\in
l_\infty$$

Now define
$$H=\{\sum_{k=1}^\infty \lambda_k y_k;(\lambda_k)\in
l_\infty\}\;\; , \;\;N=\overline{span}\{y_k\}$$ and
$$T:l_\infty \rightarrow H \;\; \mbox{by} \;\;T(\lambda_k)=
\sum_{k=1}^\infty \lambda_k y_k.$$ Suppose $\{\varphi_i\} \subset
H^\ast $ with $\varphi_i|_N=0$. Then $\psi_i=\varphi_i \circ T \in
l_\infty ^\ast$ and $\psi_i|_{c_0}=0$. Thus there is $x\in
l_\infty \backslash c_0$ such that $x\in \cap_{i=1}^\infty
\psi_i^{-1}(0)$. Consequently $T(x) \notin N$ and $T(x)\in
\cap_{i=1}^\infty \varphi_i^{-1}(0)$. Now the results follow from
Remark 16.
\end{proof}

\bigskip
To prove the main result of this paper we need the following
lemma.

\begin{lemma}Let X be a complex Banach space and $P:X\rightarrow
\mathbb{C}$ a 2-homogeneous polynomial. Suppose that $M$ is a
separable P-maximal subspace. Then there exists a countable family
$\{\varphi_i\}$ in $X^\ast $ such that $M=\cap_{i=1}^\infty
\varphi_i^{-1}(0).$
\end{lemma}
\begin{proof} Let $\{x_1,x_2,\ldots \}$ be a countable set such
that $M=\overline{\{x_1,x_2,\ldots \}}$. Thus if
$\overset{\vee}{P}$ is the symmetric bilinear form associated to P
it is clear that $M\subset \cap_{i=1}^\infty \varphi_i^{-1}(0)$
where $\varphi_i \in X^\ast$ is defined by
$$\varphi_i(x)=\overset{\vee}{P}(x_i,x)\;\;\mbox{for}\;\;i=1,2,\ldots$$ Since $M$
is a P-maximal subspace we have $M=\cap_{i=1}^\infty
\varphi_i^{-1}(0)$ or there exist $ y\in X$ such that
$\cap_{i=1}^\infty \varphi_i^{-1}(0)=M\oplus [y]$. If
$\cap_{i=1}^\infty \varphi_i^{-1}(0)=M\oplus [y]$ by the
Hahn-Banach theorem we may get $\varphi_0 \in X^\ast $ such that
$\varphi_0|_M=0$ and $\varphi_0(y)=1$ and so $M=\cap_{i=0}^\infty
\varphi_i^{-1}(0)$.
\end{proof}

\bigskip
We observe that this lemma is not true in the real case.

\begin{theorem}
Let $P:l_\infty \rightarrow \mathbb{C}$ be a 2-homogeneous
polynomial. Then there exist a non separable P-maximal subspace.
\end{theorem}

\begin{proof}
From theorem 2 there exist an infinite dimensional subspace
$M\subset c_0$ such that $P|_M=0$. Since $M$ contains a subspace
$N$ which is isomorphic to $c_0$ using the previous theorem we can
get subspaces $N_1\subset N$ and $H\subset l_\infty$ such that
$N_1\subset H $, $N_1$ being isomorphic to $c_0$ and $H$ being
isomorphic to $l_\infty$ and $H/N_1$ has the NIP. Let
$\mathcal{M}$ be a P-maximal subspace containing $N_1$. If
$\mathcal{M}$ is a separable subspace, then by lemma 19 there
exist a countable family of functional $\{\varphi_i\}$ such that
$$N_1\subset \cap_{i=1}^\infty \varphi_i^{-1}(0)=\mathcal{M}$$ a
contradiction with to the Remark 16 because $H/N_1$ has the NIP.
\end{proof}

First, by $\n$ we indicate the set $\{1,2,\ldots\}$ and by $\n _0$
we indicate the set $\{0,1,2,\ldots\}$. Given $n,j$ and $k\geq 2$
in $\n$ we define $\del{n}{j}{k}$= cardinality of the set
$\{(\cart)\in
{\n_0}^{k-1}\,\,\text{with}\,\,\sum_{h=1}^{k-1}i_h=n-j\}$ and it
is clear that $\del{n}{j}{k}>\del{n}{j}{k-1}$. With these
notations we have the following theorem.

\begin{theorem} Let $f:\n \rightarrow \n$ be a function having
	the following three properties:
	
	1) $f(1)=2$
	
	2) $f(m)>f(n)$ for every $m>n$ in $\n$.
	
	3) For every complex $f(k)$-dimensional vector space $X$ and
	$\varphi \in X^{\ast}$ there exists a $k$-dimensional subspace
	$H\subset X$ such that $\varphi|_H=0$.
	
	\medskip
	
	Let $f_1,f_2,\ldots:\n \rightarrow \n$ be the sequence of
	functions given by
	
	$f_1=f$
	
	$f_2(1)=2$ and $f_2(k)=k-1+(\stack{\del{2}{1}{k}}{\f{f_1}})(2)$ if
	$k\geq 2$
	
	$f_3(1)=2$ and $f_3(k)=k-1+(\stack{\del{3}{2}{k}}{\f{f_2}}\circ
	\stack{\del{3}{1}{k}}{\f{f_1}})(2)$
	
	$\vdots$
	
	$f_n(1)=2$ and $f_n(k)=k-1
	+(\stack{\del{n}{n-1}{k}}{\f{f_{n-1}}}\circ
	\stack{\del{n}{n-2}{k}}{\f{f_{n-2}}}\circ \ldots \circ
	\stack{\del{n}{1}{k}}{\f{f_1}})(2)$
	\bigskip
	
	Then, we have:
	\bigskip
	
	1) $f_i(k)>k$ for every $i$ and $k$ in $\n$.
	
	2) $(\stack{m}{\f{f_i}})(k)>(\stack{n}{\f{f_i}})(k)$ for every
	$m>n$ and $k$ and $i$ in $\n$.
	
	3) $f_i(m)>f_i(n)$ for every $m>n$ and $i$ in $\n$.
	
	4) Fixed $i$ and $k$ in $\n$, for every complex
	$f_i(k)$-dimensional vector space $X$ and $P\in P(^iX)$ there
	exists a $k$-dimensional subspace $H\subset X$ such that $P|_H=0$.
	
\end{theorem}

\begin{proof}
	The assertions (1) and (3) are very easy to prove by induction and
	the assertion (2) follows from (1). Thus we will prove (4) by
	induction on $i$. If $i=1$ the result is clear. Suppose that for
	$i=1,2,\ldots,n-1$ the result is true and we will prove that in
	this case the result is also true for $i=n$. We prove for $i=n$ by
	induction on $k$.
	
	If $k=1$ we have $f_n(1)=2$ and the result is clear. Now let $X$
	be a $f_n(k)$-dimensional space and $P\in P(^nX)$. From the
	assertion (3) we have that $f_n(k)>f_n(k-1)$ and thus there exists
	a $f_n(k-1)$-dimensional subspace $M\subset X$. By induction there
	exists a $(k-1)$-dimensional subspace
	$H_1=[x_1,x_2,\ldots,x_{k-1}]\subset M \subset X$ such that
	$P|_{H_1}=0$. As dim$X$=$f_n(k)=k-1
	+(\stack{\del{n}{n-1}{k}}{\f{f_{n-1}}}\circ
	\stack{\del{n}{n-2}{k}}{\f{f_{n-2}}}\circ \ldots \circ
	\stack{\del{n}{1}{k}}{\f{f_1}})(2)$ there exists a subspace
	$N\subset X$ such that $X=H_1\oplus N$ and dim$N$=$
	(\stack{\del{n}{n-1}{k}}{\f{f_{n-1}}}\circ
	\stack{\del{n}{n-2}{k}}{\f{f_{n-2}}}\circ \ldots \circ
	\stack{\del{n}{1}{k}}{\f{f_1}})(2)$. Let me prove that there
	exists a 2-dimensional subspace $H\subset N$ such that $P|_H=0$.
	For this end for each $(\cart )\in {\n_0}^{k-1}$ such that
	$i_1+\ldots +i_{k-1}=1$ we take the $(n-1)$ homogeneous polynomial
	on $N$ given by
	$$P_{\cart} (x)=A(x_1^{i_1},x_2^{i_2},\dots,x_{k-1}^{i_{k-1}},x^{n-1}).$$
	Note that we have the polynomials
	$$P_1(x)=A(x_1,x^{n-1}),P_2(x)=A(x_2,x^{n-1}),\ldots,
	P_{k-1}(x)=A(x_{k-1},x^{n-1})$$
	
	As $P_1$ is an $(n-1)$-homogeneous polynomial on $N$, by induction
	there is a subspace $N_1\subset N$ with dim$N_1$=$
	(\stack{\del{n}{n-1}{k}-1}{\f{f_{n-1}}}\circ
	\stack{\del{n}{n-2}{k}}{\f{f_{n-2}}}\circ \ldots \circ
	\stack{\del{n}{1}{k}}{\f{f_1}})(2)$ and $P_1|_{N_1}=0$. As $P_2$
	is a $(n-1)$-homogeneous polynomial on $N_1$, there exists a
	subspace $N_2\subset N_1$ such that dim$N_2$=$(\stack{\del{n}{n-1}{k}-2}{\f{f_{n-1}}}\circ
	\stack{\del{n}{n-2}{k}}{\f{f_{n-2}}}\circ \ldots \circ
	\stack{\del{n}{1}{k}}{\f{f_1}})(2)$ and $P_2|_{N_2}=0$. Continuing
	we obtain a subspace $N_{k-1}$ such that $N_{k-1}\subset
	N_{k-2}\subset \ldots\subset N_1\subset N$ and
	dim$N_{k-1}$=$\stack{\del{n}{n-2}{k}}{\f{f_{n-2}}}\circ \ldots
	\circ \stack{\del{n}{1}{k}}{\f{f_1}})(2)$ and $P_j|_{N_{k-1}}=0$
	for every $j=1,2,\ldots,k-1$. Now for each $(\cart )$ such that
	$i_1+\ldots +i_{k-1}=2$ we take the $(n-2)$ homogeneous polynomial
	on $N_{k-1}$ given by
	$$P_{\cart}
	(x)=A(x_1^{i_1},x_2^{i_2},\dots,x_{k-1}^{i_{k-1}},x^{n-2}).$$ Now
	we can use the same argument above for each $P_{\cart}$ and get a
	subspace $H_2$ such that $H_2\subset N_{k-1}$ and
	dim$H_2$=$\stack{\del{n}{n-3}{k}}{\f{f_{n-3}}}\circ \ldots \circ
	\stack{\del{n}{1}{k}}{\f{f_1}})(2)$ and $P_{\cart}|_{H_2}=0$ for
	every $\cart$ such that $i_1+\ldots +i_{k-1}=2$ or $i_1+\ldots
	+i_{k-1}=1$. Continuing we obtain a subspace $H$ such that
	$H\subset N$ and dim$H$=2 and $P_{\cart}|_H=0$ for every $\cart$
	with $i_1+\ldots +i_{k-1}=j$ with $j=1,2,\ldots,n-1$. Now take
	$y\in H$ such that $y\ne 0$ and $P(y)=0$. As $y\in H$ we have
	$$A(x_1^{i_1},\dots,x_{k-1}^{i_{k-1}},y^j)=0\,\, \mbox{for
		every}\,\, \cart \,\,\mbox{with}\,\,i_1+\ldots +i_{k-1}=n-j$$ and
	for every $j=0,1,\ldots,n$. So we have $$P|_{H_1\oplus [y]}=0$$.
\end{proof}
\begin{remark} Note that we can get a sequence as
	in the theorem [7], taking $f(k)=f_1(k)=k+1$ and in that case we
	have $f_2(k)=2k$.
\end{remark}
Now the following corollary is clear.

\begin{corollary} Let $(f_i)$ be a sequence as in the theorem [7]. let
	$X$ be a complex vector space with dim$X$=$(f_{i_1}\circ
	f_{i_2}\circ \ldots \circ f_{i_p})(k)$ and $P_1\in
	P(^{i_1}X),P_2\in P(^{i_2}X),\ldots,P_p\in P(^{i_p}X)$. Then there
	exists a k-dimensional subspace $H\subset X$ such that $P_i|_H=0$
	for every $i=1,2,\ldots,p$.
\end{corollary}

\begin{remark}
	Given a sequence $(f_i)$ as in the theorem [7] and $P\in P(^nX)$
	we associate the sequence $(k_j)$ defined by
	
	$k_1=1$
	
	$k_2$=$(f_{n-1}\circ f_{n-2}\circ \ldots f_1)(2)$
	
	$k_3$=$(\stack{\del{n}{n-1}{3}}{\f{f_{n-1}}}\circ
	\stack{\del{n}{n-2}{3}}{\f{f_{n-2}}}\circ \ldots \circ
	\stack{\del{n}{1}{3}}{\f{f_1}})(2)$
	
	\vdots
	
	$k_j$=$(\stack{\del{n}{n-1}{j}}{\f{f_{n-1}}}\circ
	\stack{\del{n}{n-2}{j}}{\f{f_{n-2}}}\circ \ldots \circ
	\stack{\del{n}{1}{j}}{\f{f_1}})(2)$
\end{remark}

Now it is simple to obtain the analogous to lemma [1] for
n-homogeneous polynomials, that is, it is simple to prove the
following result.
\begin{theorem}Let $X$ be a complex normed space and $P:X \rightarrow\mathbb{C}$ be a n-homogeneous polynomial and $A$
	be its associated symmetric n-linear form. Let $(f_i)$ be a
	sequence as in the theorem [7] and $(k_j)$ the sequence associated
	to a sequence $(f_i)$ as in the remark [10]. Then if $(e_i)$ is a
	linearly independent sequence in $X$ there is a sequence $(x_j)$
	such that
	
	$$A(x_{i_1},x_{i_2},\ldots,x_{i_n})=0\,\,\mbox{for every}\,\,
	i_1,i_2,\ldots,i_n \in \n$$ and
	
	$x_1\in [e_1,e_2]$

	$x_2\in [e_3,e_4,\ldots,e_{3+k_2-1}]$
	
	$x_3\in [e_{3+k_2},\ldots,e_{3+k_2+k_3-1}]$
	
	$\vdots$
	
	$x_j\in [e_{3+k_2+k_3+\ldots+k_{j-1}},\ldots,e_{3+k_2+k_3+k_j-1}]$
	for every $j\geq 2$ in $\n$.
\end{theorem}
\begin{proof} We obtain $(x_j)$ by induction. We start choosing
	$x_1 \in [e_1,e_2]$ such that $P(x_1)=0$. Suppose that we have
	obtained $x_1,x_2,\ldots,x_{j-1}$ and we will show how to get
	$x_j$. Let $M$ be the subspace given by
	$M=[x_1,x_2,\ldots,x_{j-1}]$ and $N$ be the subspace given by
	$N=[e_{3+k_2+k_3+\ldots+k_{j-1}},\ldots,e_{3+k_2+k_3+k_j-1}]$. For
	each $(i_1,i_2,\ldots,i_{j-1}) \in {\n _0}^{j-1}$ and $p \in
	\{1,2,\ldots,n-1\}$ with $i_1+\ldots+i_{j-1}=n-p$ we take the
	p-homogeneous polynomial $P_{i_1,\ldots,i_{j-1}}$ given by
	$$P_{i_1,\ldots,i_{j-1}}(y)=A(x_1^{i_1},\ldots,x_{j-1}^{i_{j-1}})$$
	Note that fixed $p \in \{1,2,\ldots,n-1\}$ we have exactly
	$\del{n}{p}{j}$
	p-homogeneous polynomials as above and that $N$ is a
	$k_j$-dimensional subspace with
	$k_j=\stack{\del{n}{n-1}{j}}{\f{f_{n-1}}}\circ
	\stack{\del{n}{n-2}{j}}{\f{f_{n-2}}}\circ \ldots \circ
	\stack{\del{n}{1}{j}}{\f{f_1}})(2)$. From corollary [9] we have a
	2-dimensional subspace $H\subset N$ such that
	$$P_{i_1,\ldots,i_{j-1}}|_H=0$$ for every $(i_1,i_2,\ldots,i_{j-1}) \in {\n _0}^{j-1}$ and $p \in
	\{1,2,\ldots,n-1\}$ with $i_1+\ldots+i_{j-1}=n-p$. Now we take
	$x_j \in H$ such that $x_j\ne 0$ and $P(x_j)=0$ and the result
	follows.
\end{proof}

Now the proofs of theorems [12] and [13] are analogous to the
proofs of theorems [2] and [6] respectively.

\begin{theorem}Let $X$ be a normed complex space containing a subspace $M$
	isomorphic to $c_0$. If $P:X \rightarrow\mathbb{C}$ is a
	n-homogeneous polynomial, then there is a subspace $M_P$ such that
	$P|_{M_P}=0$ and $M_P$ is isomorphic to $c_0$.
\end{theorem}
\begin{theorem}
	Let $P:l_\infty \rightarrow \mathbb{C}$ be a n-homogeneous
	polynomial. Then there is a subspace $H\subset l_\infty$ such that
	$P|_H=0$ and $H$ is isomorphic to $l_\infty$. In particular $H$ is
	a non-separable space.
\end{theorem}

\end{document}